\documentclass[12pt]{amsart}
\usepackage[left=2.7cm,right=2.7cm,top=3.5cm,bottom=3cm]{geometry}
\usepackage{amssymb,amsmath,comment,mathrsfs,mathtools,tikz-cd,todonotes}
\usepackage[colorlinks=true,linkcolor=blue]{hyperref}
\usepackage{cleveref}

\theoremstyle{remark}

\newtheorem{para}{\bf}[subsection]

\newtheorem{rem}[para]{\bf Remark}

\theoremstyle{definition}

\newtheorem{dfn}[para]{Definition}

\theoremstyle{plain}

\newtheorem{theorem}[para]{Theorem}

\newtheorem{lemma}[para]{Lemma}
\newtheorem{cor}[para]{Corollary}

\newtheorem{prop}[para]{Proposition}

\newcommand{\bbF}{{\mathbb F}}
\newcommand{\bbG}{{\mathbb G}}

\newcommand{\bbN}{{\mathbb N}}

\newcommand{\bbQ}{{\mathbb Q}}

\newcommand{\bG}{{\bf G}}

\newcommand{\bN}{{\bf N}}

\newcommand{\bb}{{\bf b}}

\newcommand{\frm}{{\mathfrak m}}

\newcommand{\cA}{{\mathcal A}}
\newcommand{\cB}{{\mathcal B}}

\newcommand{\cM}{{\mathcal M}}
\newcommand{\cN}{{\mathcal N}}
\newcommand{\cO}{{\mathcal O}}

\newcommand{\cR}{{\mathcal R}}

\newcommand{\Fp}{{\bbF_p}}

\newcommand{\Fil}{{\rm Fil}}

\newcommand{\gr}{{\rm gr}^\bullet}

\newcommand{\height}{{\rm ht}}

\renewcommand{\mod}{\mathrm{mod}\text{ }}

\newcommand{\ovcA}{\overline{\cA}}
\newcommand{\ovcB}{\overline{\cB}}
\newcommand{\ovcR}{\overline{\cR}}
\newcommand{\ot}{\otimes}

\newcommand{\Q}{{\mathbb Q}}

\newcommand{\Qp}{{\bbQ_p}}

\newcommand{\ra}{\rightarrow}

\newcommand{\Z}{{\mathbb Z}}

\newcommand{\Zp}{{\mathbb Z_p}}

\newcommand{\I}{{\mathrm{I}}}

\keywords{Iwasawa algebra, pro-$p$ Iwahori, reductive groups}
\subjclass[2020]{Primary: 11R23, 20G25  Secondary: 16L30. }
\begin{document}

\title{Presentation of an Iwasawa algebra:\\The pro-$p$-Iwahori of  reductive groups }

\author{Aranya Lahiri}
\address[Lahiri]{University of California San Diego}
\email{arlahiri@ucsd.edu}

\author{Jishnu Ray}

\address[Ray]{Harish Chandra Research Institute, Chhatnag Road, Jhunsi, Prayagraj (Allahabad) 211 019 India}
\email{jishnuray@hri.res.in; jishnuray1992@gmail.com}

\maketitle

\begin{abstract}
    In this article we generalize results of Clozel and Ray (for $SL_2$ and $SL_n$ respectively) to give explicit ring-theoretic presentation in terms of a complete set of generators and relations of the Iwasawa algebra of the pro-$p$ Iwahori subgroup of a connected, split, reductive group $\bbG$ over $\Qp$.
\end{abstract}

\section{Introduction} 
In this article we are interested in the  Iwasawa algebra of the pro-$p$ Iwahori subgroup of the $\Qp$-valued points $G:=\bG(\Qp)$, of a connected reductive group $\bG$ defined and split over $\Qp$.

\vskip8pt

Let's recall that for a pro-$p$ group $G$ and a finite extension $L/\Qp$ with ring of integers $\cO$, the  projective limit over the set of open normal subgroups $\cN(G)$ of $G$, 
\[\Lambda(G) := \cO[[G]] := \varprojlim_{N\in \cN(G)}\cO[G/N],\]
equipped with the appropriate projective limit topology, is called the \textit{completed group ring} or the \textit{Iwasawa algebra} of $G$ over $\cO$. It is a complete pseudo-compact topological ring that plays an important role in the study of representations of $p$-adic groups and number theory via it's connection to the $p$-adic Local Langlands program. The reader may find in the  excellent survey article \cite{Ard06} a collection of several ring theoretic properties of Iwasawa algebras of a $p$-adic group and  several open problems on them.

\vskip8pt 
In this article we give a ring theoretic presentation of $\Lambda(I)$, where $I$ is the pro-$p$ Iwahori subgroup of a connected reductive group $G$ satisfying the assumptions above. We work under the assumption that $p>h+1$ where $h$ is the Coxeter number of the reductive group (see Section \ref{sec:pre}). Under this assumption the pro-$p$ Iwahori $I$ is $p$-saturated in the sense of Lazard (see Theorem \ref{Ordered basis}, \cite{L-S}). The main result of the article is,

\begin{theorem}
For $p>h+1$, the Iwasawa algebra $\Lambda(I)$ is naturally isomorphic as topological rings to $\cA/\cR$ (see Theorem \ref{Main isomorphism of iwasawa algebras}).
    
\end{theorem}

Here $\cA$ is the noncommutative power series ring over $\Z_p$ in a certain number of variables (depending on $G$) and $\cR$ is a two-sided ideal of $\cA$ coming from the Chevalley relations on the $p$-adic group $G$.

\vskip8pt

The first result in this direction was carried out by  Clozel in \cite{Clo11},  then taken up in \cite{Clo17} and used in \cite{Clo18} in context of a proposed `$p$-adic base change map'. There he gives an explicit ring theoretic presentation of the Iwasawa algebra $\Lambda(I_\Qp)$ for the pro-$p$ Iwahori subgroup $I_\Qp$ of $SL_2(\Qp)$ and of $\Lambda(I_L)$, the pro-$p$ Iwahori subgroup $I_L$ of $SL_2(L)$ for some unramified extension $L/\Qp$. Later, the second author of this article  generalized the result to the pro-$p$ Iwahori subgroup of $SL_n(\Qp)$ (see \cite{Ray20}). He also solved the case of a general uniform pro-$p$ group in \cite{Ray19}.

\vskip8pt

Here we follow the circle of ideas of Clozel and the second authors' work for $SL_2$ and $SL_n$. The main technical challenge here is that the pro-$p$ Iwahori subgroup can be equipped with a $p$-valuation with respect to which it is $p$-saturated but in general it is not a uniform group. Essentially, the main difference being the associated graded algebra in the case of a uniform pro-$p$ group is a commutative polynomial ring, while that is not the case in our setup. 

\vskip8pt

Clozel could carry out his calculations by hand because of a relatively small number of variables and the relations between them being simple and easily calculable. The  techniques used by the second author in \cite{Ray20} were also explicit but involved somewhat tedious computations involving large matrices.

\vskip8pt

Clozel's main  interest in the explicit ring theoretic presentation of the Iwasawa algebras seem to stem from the fact  that he could propose a formal candidate for the $p$-adic base change map between  $\Lambda(I_L)\ra \Lambda(I_\Qp)$. But even Clozel himself notes that the right setup for such a base change map perhaps involves the rigid analytic distribution algebras associated to the corresponding rigid analytic groups of the pro-$p$ Iwahori subgroups. And as of now, we are not sure of the significance of such a formal base change map and thus for the purpose of exposition and ease of calculations we have restricted ourselves to groups defined over $\Qp$. The same technique should generalize to groups defined over an unramified extension $L/\Qp$ as well, and using the explicit presentation it should not be difficult to define a `formal base change map' between $\Lambda(I_L)\ra\Lambda(I_\Qp)$.

\vskip8pt

\subsection{Sketch of ideas for the main theorem}
Before starting to write things formally we give a brief and slightly vague account of the main ideas going into the proof of the main theorem.

\vskip8pt
The natural $p$-valuation on the pro-$p$ Iwahori subgroup $I$ of $G$ makes it a $p$-saturated group, which in turn implies that $I$ is homeomorphic to $\mathbb{Z}_p^d$ as a $\Qp$-analytic manifold. Thus the Iwasawa algebra $\Lambda(I)$ is isomorphic as a $\Zp$-module to $\Zp[[X_1,\cdots,X_d]]$. Now the Chevalley relations between the Chevalley basis of the Lie algebra of $I$ induces relations between the variables $X_i$. We show here that this is a complete set of generator and relations.

\vskip8pt

A key observation towards the proof is to realize that it is enough to prove an analogous statement for the mod-$p$ Iwasawa algebra $\Omega(I):=\Lambda(I)\ot_\Zp \Fp$. Reduction mod-$p$ makes the Chevalley relations much easier to work with. The proof then boils down to a dimension counting of associated graded vector space of $\Omega(I)$, as executed in Theorem \ref{dimension bounding}.

We point out some useful context and contrast for the reader of this article: 
\begin{itemize}
\item By using results of a previous paper by the first author and Sorensen (\cite{L-S}) we conceptualize the calculations significantly. The key input being a systematic way of writing down a $p$-valuation and an ordered basis for the pro-$p$ Iwahori subgroup. This was missing in literature when the second author published \cite{Ray20}, which caused him to resort to the long calculations. This simplification also allows us to deal with the Chevalley relations without having to break them into too many cases. 

\item We replace the lexicographic ordering on the set of positive roots  in \cite{Ray20} by any additive ordering as constructed by Papi in \cite{Papi94}. In the process perhaps we shade more light on the seemingly adhoc choice of order in the ordered basis for $I$ in \cite{Ray20}.

\item We also provide a more conceptual proof of Theorem \ref{dimension bounding}, which is in fact the technical heart of the article. 

\item In \cite{Eri21} the author also gives a ring theoretic presentation of the mod-$p$ Iwasawa algebra. But their methods are very different from ours, and we do not know the connection between these two approaches yet.

\end{itemize}

\subsection{Applications in Iwasawa theory}
Let $E$ be an elliptic curve without complex multiplication, over a number field $F$ with good ordinary reduction at places of $F$ above $p$. Consider the \textit{trivializing extension} $F_\infty:=F[E_{p^\infty}]=\cup_{n \geq 1}F[E_{p^n}]$, where $E_{p^n}$ is the $p^n$-torsion points of $E(\overline{F})$.  This trivializing extension has Galois group $G$ which is a compact open subgroup of $GL(2,\Z_p)$ for all primes $p$; and  $G=GL(2,\Z_p)$ for all but finitely many primes $p$. The Pontryagin dual $\mathcal{X}_\infty$ of the Selmer group $\widehat{\mathrm{Sel}(E/F_\infty)}$ over this trivializing extension carries several arithmetic information of $E$ and its structure as an Iwasawa module has deep number theoretic consequences. It is conjectured (see  \cite[Conjecture 1.7]{Coates_Fragments}) that $\mathcal{X}_\infty$ is a finitely generated  torsion $\Lambda(G)$-module. See Theorem 1.10 and the example following it of \textit{(loc.cit)} for cases 
when the conjecture holds. A partial result giving the explicit structure of this dual Selmer group as an Iwasawa module is also known. More precisely, $\mathcal{X}_\infty$ is pseudo-isomorphic to a direct sum of cyclic Iwasawa modules:
$$\mathcal{X}_\infty \sim \oplus_{i=1}^m \Lambda(G)/L_i,$$
where $L_i$ are \textit{left} reflexive ideals in the Iwasawa algebra $\Lambda(G)$ (see the structure theorem in \cite[pg. 74]{CoatesSchneiderSujatha-modules}). The explicit description of these \textit{left} reflexive ideals is not known. So a natural question is whether one can use the explicit presentation of the Iwasawa algebra $\Lambda(G)$ to explicitly describe these reflexive ideals that are of arithmetic interest. Several progress has been made in the literature regarding the nonexistence of the \textit{two-sided} reflexive ideals and over the mod-$p$ Iwasawa algebra $\Omega(G)$. An element $x \in \Omega(G)$ is called \textit{normal} if $x\Omega(G)=\Omega(G)x$ and the ideal generated by a normal element is a two-sided reflexive ideal of $\Omega(G)$. When $G$ is the first congruence kernel of a split, semi-simple, simply connected Chevalley group over $\Z_p$, under certain hypothesis, the  explicit presentation of its Iwasawa algebra given by the second author in  \cite{Ray2018algebra} can be used  to show that every nonzero normal element of $\Omega(G)$ must be a unit (see \cite[Theorem 5.1]{HDR}). Hence there cannot be any two-sided non-trivial reflexive ideal generated by a normal element in the mod-$p$ Iwasawa algebra. Note that this technique of using the explicit presentation to prove such a result on the normal element (\textit{loc.cit}) also works for the prime $p=3$ and hence generalizes an earlier result  by Ardakov, Wei and Jhang; their techniques were also entirely different (see \cite[Theorem A]{AWJ}).
Such an explicit presentation of Iwasawa algebra  could also be used to reprove known results regarding the center of the mod-$p$ Iwasawa algebra (see \cite[Proposition 7.1]{HDR}). We must mention that the techniques in \cite{HDR} concerns only the first congruence kernel and hence the underlying group is uniform pro-$p$. It is a natural question, which we believe should have an affirmative answer, to generalize the results in \textit{(loc.cit)} and investigate whether the explicit presentation of the pro-$p$ Iwahori subgroup presented in this paper could be used to explore one-sided or two-sided reflexive ideals of the correponding Iwasawa algebra.

\section*{Acknowledgement}
The authors thank Peter Schneider and Claus Sorensen for several discussions throughout this project. The second author is supported by the Inspire research grant.
\section{Preliminaries}

\subsection{Iwahori factorization}\label{sec:pre}
 Let $p$ be a prime. For the sake of clarity we will discuss everything with the assumption that our ground field is $\Qp$ and just note at the beginning that most of the results in this preliminary section are usually stated in their corresponding sources in an appropriately generalized fashion for a finite extension $L/\Qp$. 
 
 We fix a connected reductive group ${\bf{G}}$ which is split over $\Qp$, and choose an $\Qp$-split maximal torus ${\bf{T}}\subset {\bf{G}}$. Let 
 $(X^*({\bf{T}}), \Phi, X_*({\bf{T}}), \Phi^{\vee})$ be the associated root datum. We fix a system of positive roots $\Phi^+$ with simple roots 
 $\Delta \subset \Phi^+$ and let $\Phi^-=-\Phi^+$. We let $\text{ht}(\cdot)$ be the height function on $\Phi$, and we recall that the 
 Coxeter number of ${\bf{G}}$ is $h=1+\text{ht}(\theta)$ where $\theta$ is the highest root.  
 
Let  $G={\bf{G}}(\Qp)$ with the usual locally profinite topology, and similarly $T={\bf{T}}(\Qp)$. Further we will denote the maximal compact subgroup by $T^0\subset T$ and its Sylow pro-$p$-subgroup by $T^1$. Recall that to each $\alpha\in \Phi$ one can attach a unipotent root subgroup 
 ${\bf{U}}_{\alpha}$ normalized by ${\bf{T}}$. We let $U_{\alpha}={\bf{U}}_{\alpha}(\Qp)$ and note that it comes with a natural isomorphism 
 $u_{\alpha}: \Qp \overset{\sim}{\longrightarrow} U_{\alpha}$ such that $tu_{\alpha}(x)t^{-1}=u_{\alpha}(\alpha(t)x)$ for all $t \in T$ and $x \in \Qp$. This gives a filtration of $U_{\alpha}$ by the subgroups $U_{\alpha,r}=u_{\alpha}((p^r))$ for varying $r \in \Z$.

We consider the full Iwahori subgroup $J$ corresponding to $\Phi^+$. If $\bG$ is assumed to be semisimple and simply connected then $J$ has the geometric interpretation of being the stabilizer of the maximal facet (or alcove) corresponding to the root datum of the Bruhat-Tits building $\cB(G)$ of $\bG$. We don't pursue the geometric aspects of $J$ and instead just note that the so-called Iwahori factorization says that multiplication defines a homeomorphism,
$$
\prod_{\alpha\in \Phi^-}U_{\alpha,1} \times T^0 \times \prod_{\alpha\in \Phi^+}U_{\alpha,0} \overset{\sim}{\longrightarrow} J.
$$
We concentrate our attention on the Sylow pro-$p$ subgroup of $J$, namely the pro-$p$ Iwahori subgroup $I$. Again multiplication gives a homeomorphism,
\begin{equation}\label{iwahori}
\prod_{\alpha\in \Phi^-}U_{\alpha,1} \times T^1 \times \prod_{\alpha\in \Phi^+}U_{\alpha,0} \overset{\sim}{\longrightarrow} I.
\end{equation}
Here the products over $\Phi^-$ and $\Phi^+$ are ordered in an arbitrarily chosen way,  see the proof of \cite[Lem.~2.3]{OS19} and Lemma 2.1, part i, in \cite{OS19}. From Section \ref{Ordered basis of pro-$p$ Iwahori}, we will fix an ordering of this product to discuss ordered basis and subsequent developments that depend on the ordered basis.

\subsection{$p$-saturated groups}
To make the article somewhat self-contained we collect here a few definitions related to $p$-valuation on a group $G$ which can be any pro-$p$, torsion free, compact $p$-adic Lie group. Let us recall that a \textbf{$p$-valuation} \cite[Chapter 23, pg 169]{Schneider_Lie} $\omega$ on the group $G$ is a real valued function, \\
$\omega : G \setminus \{1\} \ra (\frac{1}{p-1},\infty)$ with the convention $\omega(1) = \infty$, satisfying
\begin{enumerate}
    \item $\omega(g^{-1}h) \ge \min(\omega(g),\omega(h))$,
    \item $\omega([g,h])\ge \omega(g)+ \omega(h)$,
    \item $\omega(g^p)= \omega(g)+1$.
\end{enumerate}

\vskip8pt

A sequence of elements $(g_1, \cdots, g_r)$ in a group $G$ equipped with a $p$-valuation $\omega$ is called an \textbf{ordered
basis} \cite[Def, Chapter 26, pg 182]{Schneider_Lie} of $(G, \omega)$ if it satisfies

\begin{enumerate}
    \item $\mathbb{Z}_p^d \ra G, (x_1,\cdots, x_d) \ra g_1^{x_1}\cdots g_d^{x_d}$ is a homeomorphism,
    \item  $\omega(g_1^{x_1}\cdots g_d^{x_d})= \min_{1\le i\le d} (\omega(g_i)+v_p(x_i))$ for any $x_1,\cdots, x_d\in \Zp$.
\end{enumerate}
\vskip8pt
Finally a $p$-valuable group $(G,\omega)$  equipped with the valuation $\omega$ is \textbf{saturated} \cite[Def, chapter 26, pg 187]{Schneider_Lie} if  any $g$ with $\omega(g) > \frac{p}{p-1}$ is a $p$-th power of some element of $G$.
\subsection{Additive ordering on the set of positive roots}\label{sec:additive}
We recollect results from \cite{Papi94} in a form we will need to prove Lemma \ref{Counting inversions} later.

\vskip8pt

Let $W$ be the associated Weyl group to the root system $\Phi$. Let $\Delta=\{\delta_1,\cdots, \delta_n\}$ be a set of simple roots and $\Phi^+$ cooresponding set of positive roots and $\Phi^-=-\Phi^+$. And let $s_1,\cdots, s_n$ be the reflections corresponding to simple roots.

\begin{dfn}\label{associate to w }
If $S:=\{\sigma_1,\cdots,\sigma_r\}\subset \Phi^{+}$ is a subset of positive roots of $\Phi$ we say that $S$ is associated to $w\in W$ if there is a reduced expression $w= s_{i_1}\cdots s_{i_r}$ of $w$ such that $\sigma_1=\delta_{i_1}, \sigma_j= s_{i_1}\cdots s_{i_{j-1}}(\delta_{i_j})$, $\delta_{i_j}\in \Delta$.
\end{dfn}

In particular note that, if $w_0$ is the longest element of the Weyl group then $\Phi^{+}$ is associated to $w_0$.
\begin{dfn}\label{compatible ordering }
 An ordered subset $(S, <)$ of $\Phi^{+}$ is said to be \textbf{compatibly ordered} if the following two conditions hold.
 \vskip8pt
 (1) If $\lambda,\mu\in S$ with $\lambda < \mu$, and $\lambda+\mu \in \Phi$ then $\lambda+\mu \in S$ and $\lambda < \lambda+\mu < \mu$. 

 \vskip8pt

 (2) If $ \lambda+\mu\in S$, $\lambda,\mu \in \Phi^{+}$ then $\lambda$ or $\mu$ (or both) belong to $S$ and one of them precedes $\lambda+ \mu$.
\end{dfn}

The main result of \cite{Papi94} is that 
\begin{theorem}[Theorem in pg 662 of \cite{Papi94}]\label{PapiMain}
A subset $S\subset \Phi^+$ can be given a compatible ordering if and only if it is associated to some $w$ of $W$.
\end{theorem}

Since we already observed that $\Phi^+$ is associated to $w_0$, as a corollary of \ref{PapiMain} we get

\begin{cor}\label{Papi}
We can impose a compatible ordering on $\Phi^+$.
\end{cor}

\subsection{Ordered basis of pro-$p$ Iwahori}\label{Ordered basis of pro-$p$ Iwahori} We recall what we need from \cite[Section 3]{L-S} with slight modification and without any proofs.
\vskip8pt
In \cite{L-S} it was shown that the pro-$p$ Iwahori subgroup $I$ can be equipped with a $p$-valuation $\omega$ with respect to which $I$ is saturated. For us it will be important to work with a choice of an ordered basis for $(I, \omega)$ and to be able to explicitly write down the valuations of these basis elements.

\vskip8pt

Let us fix any ordering on the elements of $\Phi^{-}$ with increasing height function on the roots, the elements of $\Delta$ in an arbitrary way and the elements of $\Phi^+$ with respect to an arbitrary but fixed compatible ordering in the sense of Definition \ref{compatible ordering }. And finally we combine this ordering to a totally ordered $\Phi$ in a way such that 
for any $\beta\in \Phi^-, \delta\in \Delta, \alpha\in \Phi^+$ we have $\beta < \delta < \alpha$.

\begin{theorem}\cite[Proposition 3.4, Corollary 3.6]{L-S}\label{Ordered basis}  With $p-1 > h,$  an  ordered basis with their respective valuations for $(I, \omega)$ is given by

\[\begin{array}{ccccccccccc}
   &\bullet&  &u_\beta(p)_{\beta\in \Phi^{-}}&, &\omega(u_\beta(p)) &=& 1+\frac{{\rm ht}(\beta)}{h},&\\
    &\bullet&  &\delta^\vee(1+p)_{\delta \in \Delta}&, &\omega(\delta^\vee(1+p))&=& 1,&\\
    &\bullet&  &u_\alpha(1)_{\alpha\in \Phi^+},&, &\omega(u_\alpha(1))&=& \frac{{\rm ht}(\alpha)}{h}.&
\end{array}\]

where we order the elements of the ordered basis corresponding to  a total ordering of $\Phi$ chosen as above.
    
\end{theorem}

Note that the description given in \cite{L-S} chooses $e^p$ as an ordered basis of $1+ (p)$. Here to make the  calculations more straight forward the natural choice seems to be $1+p$.  From this point onward we change our notation to write $h_\delta$ instead of $\delta^\vee$ to match with the notation used in \cite{Ray20}.

\begin{rem}\label{bound on valuation of basis}
For future reference we note that for any element $g$ of the ordered basis just mentioned above we have $\frac{1}{h}\le \omega(g)\le 1.$
    
\end{rem}

\subsection{Iwasawa Algebra}

Let $G$ be any profinite group and let $\cN(G)$ the set of open normal subgroups in $G$. Then we know that 
 \[G = \varprojlim_{N\in \cN(G)}G/N\]
is the projective limit, as a topological group, of the finite groups $G/N$ equipped witb the discrete topology.
By functoriality of associated group rings, the algebraic group rings $\cO[G/N]$ form, for varying $N$ in $\cN(G)$, a projective system of rings. The  projective limit
\[\Lambda(G) := \cO[[G]] := \varprojlim_{N\in \cN(G)}\cO[G/N]\]
is called the \textit{completed group ring} or the \textit{Iwasawa algebra} of $G$ over $\cO$. Equipping $\Lambda(G)$ with the projective limit topology of the natural topologies   of $\cO[G/N]$ makes it a  complete pseudo-compact topological ring. 

The embedding of $G \ra \Lambda(G)$ given by $g\ra \delta_g$ promotes to an injective embedding of the group ring $\cO[G]$  into $\Lambda(G)$ with dense image.

\vskip8pt
Now let $(G, \omega)$ be a $p$-saturated group over $\Z_p$ with a $p$-valuation $\omega$ and a choice of ordered basis $(h_1,\cdots, h_d)$. In this situation we get a homeomorphism of analytic manifolds 

\begin{align*}
     c:\mathbb{Z}_p^d &\rightarrow G  \\
    (x_1,\cdots, x_d) &\mapsto h_1^{x_1}\cdots h_d^{x_d}.
\end{align*} 

We use the standard notation $b_j= h_j-1\in \Lambda(G)$, and for $\alpha=(\alpha_1,\cdots,\alpha_d)\in \bN^d$  we let $\bb^\alpha= b_1^{\alpha_1}\cdots b_d^{\alpha_d}$. Then any element $\lambda\in \Lambda(G)$ can be written as $\lambda:= \sum_\alpha c_\alpha {\bf b}^\alpha$ with $c_\alpha\in \cO$ and we can naturally equip it with the valuation,
\begin{equation}\label{eq:evaluate}
\tilde{\omega}_\Lambda(\sum_\alpha c_\alpha {\bf b}^\alpha)= \inf_\alpha(v(c_\alpha)+\sum_{i=1}^d \alpha_i\omega(h_i) ).
\end{equation}

After pulling back and dualizing we get an isomorphism of topological $\Zp$-modules
\begin{align*}
  c_\star: \Lambda(\mathbb{Z}_p^d)    &\cong \Lambda(G)  \\
    \delta_{e_i} &\ra \delta_{h_i}.
\end{align*}
By \cite[20.1]{Schneider_Lie} we know that $\Lambda(\mathbb{Z}_p^d) \cong \Zp[[X_1,\cdots, X_d]]$  the power series ring in $d$ variables, as a topological ring.

Using \ref{Ordered basis} we thus rewrite for $I$,
\begin{lemma}[Lazard]\label{Zp module isomorphism of Iwasawa algebra} The following are isomorphic as $\Zp$-modules with an isomorphism given by, 
\begin{align*}
    \tilde{c}: \Zp[[U_\beta, V_\alpha, W_\delta: \beta\in \Phi^{-}, \alpha\in \Phi^{+}, \delta\in \Delta]] \cong \Lambda(I)\\
    1+ V_\alpha \ra u_\alpha(1),\\
    1+ W_\delta \ra h_\delta(1+p),\\
    1+ U_\beta \ra u_\beta(p).
\end{align*}
    
\end{lemma}

\subsection{The graded structure}
For any $p$-valued group $G$, the valuation $\tilde{\omega}$ induces a natural filtration of $\Lambda(G)$ as follows. Define for $v \ge 0$,
\[\begin{array}{cccc}
    \Lambda(G)_v &:=\big\{\lambda\in \Lambda(G)\big | \tilde{\omega}(\lambda)\ge v \big\}, \\
    \Lambda(G)_{v^+} &:=\big\{\lambda\in \Lambda(G)\big | \tilde{\omega}(\lambda)> v \big\}.
\end{array}\]

This gives  a filtration and the  associated graded algebra
\[{\rm gr}^v\Lambda(G):= \Lambda(G)_v/\Lambda(G)_{v^+}; \hspace{.2cm} {\rm gr}\Lambda(G)=\bigoplus_{v\ge 0}{\rm gr}^v\Lambda(G).   \]

Under this filtration $\Lambda(I)$ is filtered by $\frac{1}{h}\bbN$.

\vskip8pt

This follows from the fact that, ${\rm gr}^v\Lambda(I)$ is generated by the independent elements 

$p^s\prod_{\beta\in \Phi^{-}} U_\beta^{n_\beta}\prod_{\delta\in \Delta} W_\delta^{n_\delta}\prod_{\alpha\in \Phi^{+}}V_{\alpha}^{n_\alpha}$ such that
\[s+ \sum_{\beta\in \Phi^{-}} \frac{h+{{\rm ht}(\beta)}}{h}n_\beta+ \sum_{\delta\in \Delta} \frac{1}{h} n_\delta+\sum_{\alpha\in \Phi^{+}}\frac{{{\rm ht}(\alpha)}}{h}n_\alpha= v.\]

The filtration on $\Lambda(I)$ induces a natural filtration and thus a grading on $\Omega(I)= \Lambda(I)\ot_\Zp \Fp$. We will use this to compute the dimension of ${\rm gr}^v\Omega(I)$ as an $\mathbb{F}_p$-vector space. 
We know that ${\rm gr}^v\Omega(I)$ is generated by $\prod_{\beta\in \Phi^{-}} U_\beta^{n_\beta}\prod_{\delta\in \Delta} W_\delta^{n_\delta}\prod_{\alpha\in \Phi^{+}}V_{\alpha}^{n_\alpha}$ with 

\[\sum_{\beta\in \Phi^{-}} \frac{h+{{\rm ht}(\beta)}}{h}n_\beta+ \sum_{\delta\in \Delta} \frac{1}{h} n_\delta+\sum_{\alpha\in \Phi^{+}}\frac{{{\rm ht}(\alpha)}}{h}n_\alpha= v.\]

\vskip8pt 

We don't change the filtered pieces if we replace $v\in \frac{1}{h}\bbN$ by $m= vh$. This change assigns the following valuations (which we still denote by $\tilde{\omega}$ by abuse of notation) to the generators of the Iwasawa algebra.

\[\begin{array}{ccccccccccc}
  &\tilde{\omega}(u_\beta(p))&=& h+{{\rm ht}(\beta)},&\\
&\tilde{\omega}((h_\delta(1+p))&=& h,&\\
 &\tilde{\omega}(u_\alpha(1))&=& {{\rm ht}(\alpha)}.&
\end{array}\]

The remarks above recover the following fact which is really due to Lazard.
\vskip8pt

\begin{theorem}\label{Dimension of graded iwasawa}
  The dimension of ${\rm gr}^m\Omega(I)$ as an $\Fp$-vector space is equal to the   dimension of space of homogeneous symmetric polynomials of degree $m$ in the variables $\{U_\beta, V_\alpha, W_\delta: \beta\in \Phi^{-}, \alpha\in \Phi^{+}, \delta\in \Delta\}$ where the variables are assigned the degrees equal to the corresponding shifted valuations of the ordered basis, $\deg(U_\beta) = h+{{\rm ht}(\beta)}, \deg(W_\delta) = h$ and $\deg(V_\alpha) = {{\rm ht}(\alpha)}$.
\end{theorem}

\section{Relations in the Iwasawa Algebra}
\subsection{The Chevalley relations}
The Chevalley relations on the Chevalley basis of $I$ translate to relations between the non-commutative variables defining the Iwasawa algebra. As remarked in the introduction the main result of this article is  that these are a defining set of relations  for the Iwasawa algebra.

\vskip8pt
The Chevalley relations are the following (see \cite[pg 23]{Steinberg16}):
\\

\noindent (\textbf{Diag})  $h_\alpha(u) h_\beta(t)= h_\beta(t) h_\alpha(u).$
\\
\noindent (\textbf{Diag-uni}) $h_\delta(t) x_\alpha(u) h_\delta(t)^{-1}= x_\alpha(t^{\langle \alpha, \delta \rangle} u)$ where $\langle \alpha, \delta \rangle \in \Z.$
\\
\noindent (\textbf{Uni-1})
If $\alpha+\beta \neq 0$ and $\alpha+\beta \notin \Phi$, then we have 
$x_{\alpha_1}(t) x_{\alpha_2}(u)= x_{\alpha_2}(u)x_{\alpha_1}(t).$
\\
\noindent (\textbf{Uni-2})
If $\alpha_1+\alpha_2 \neq 0$ and $\alpha_1+\alpha_2 \in \Phi$, then we have $x_\alpha(t)x_\beta(u)= \prod_{i, j > 0} x_{i\alpha+ j\beta}(c_{ij} t^i u^j)x_\beta(u)x_\alpha(t),$ where $c_{ij}$'s are unique integers depending on $\alpha$ and $\beta$.
\\
\noindent (\textbf{Uni-3})
Let $\alpha\in \Phi^+$. Writing $\alpha=\sum_{\delta_i\in \Delta}n_i\delta_i$, we have $h_\alpha(t)= \prod_{\delta_i\in \Delta}h_{\delta_i}(t)^{n_i}$ for some non-negative integers $n_i$.  We obtain
    $$x_\alpha(1) x_{-\alpha}(p)= x_{-\alpha}(p)^Q\prod_{\delta_i\in \Delta}h_{\delta_i}(1+p)^{n_i}x_{\alpha}(1)^Q,$$
    where $Q=(1+p)^{-1}.$

We include a proof of \textbf{(Uni-3)}; the other relations are exactly as in \cite{Steinberg16}.  For any root $\alpha \in \Phi^+$, we have a homomorphism 
$$\phi_\alpha: SL_2(\Q_p) \rightarrow \langle x_{\alpha}(t), x_{-\alpha}(t); t \in \Q_p\rangle$$ such that 
\begin{align*}
    \phi_\alpha(1+tE_{12})&=x_\alpha(t),\\
    \phi_\alpha(1+tE_{21})&=x_{-\alpha}(t),\\
    \phi_{\alpha}(tE_{11}+t^{-1}E_{22})&=h_\alpha(t),
\end{align*}
where $E_{ij}$ is the $2 \times 2$ matrix with 1 in the $(i,j)$-th entry and zero elsewhere.
In $SL_2(\Z_p)$ we have the following matrix relation
$$(1+E_{12})(1+pE_{21})=\big(1+p(1+p)^{-1}E_{21}\big)\big((1+p)E_{11}+(1+p)^{-1}E_{22}\big)\big(1+(1+p)^{-1}E_{12}\big).$$
Since $\phi_{\alpha}$ is a homomorphism we obtain
$$x_{\alpha}(1)x_{-\alpha}(p)=x_{-\alpha}(p(1+p)^{-1})h_{\alpha}(1+p)x_\alpha((1+p)^{-1}),$$ 
which implies 
$$x_{\alpha}(1)x_{-\alpha}(p)= x_{-\alpha}(p)^{(1+p)^{-1}}\prod_{\delta_i\in \Delta}h_{\delta_i}(1+p)^{n_i}x_{\alpha}(1)^{(1+p)^{-1}}.$$

\subsection{Relations in the Iwasawa algebra}\label{relations}
The corresponding relations in the Iwasawa algebra are

\noindent (\textbf{R-1}) 
$(1+ W_{\delta_1})(1+ W_{\delta_2})= (1+ W_{\delta_2})(1+ W_{\delta_1}).$
\\
\noindent
 (\textbf{R-2})
$(1+ W_{\delta})(1+ V_\alpha)= (1+ V_\alpha)^M(1+ W_{\delta}), \alpha\in \Phi^{+}, M= (1+p)^{\langle \alpha, \delta\rangle}.$
\\
\noindent
 (\textbf{R-3})
 $(1+ W_{\delta})(1+ U_\beta)= (1+ U_\beta)^M(1+ W_{\delta}), \beta\in \Phi^{-}, M= (1+p)^{\langle \beta, \delta\rangle}.$
 \\
 \noindent
 (\textbf{R-4}) $V_\alpha U_\beta = U_\beta V_\alpha, \alpha \in \Phi^{+}, \beta \in \Phi^{-},\alpha+\beta \neq 0,\alpha+\beta \notin \Phi.$
 \\
 \noindent
 (\textbf{R-5})
 For $\alpha_1, \alpha_2 \in \Phi^{+}$,
    $(1+V_{\alpha_1})(1+V_{\alpha_2})=\underset{\stackrel{i, j>0}{i\alpha_1+j\alpha_2 \in \Phi}}{\prod}(1+V_{i\alpha_1+j\alpha_2})^{c_{ij}}(1+V_{\alpha_2})(1+V_{\alpha_1}).$
    \\
 \noindent
 (\textbf{R-6})   
For $\beta_1, \beta_2 \in \Phi^{-}$,
    $(1+U_{\beta_1})(1+U_{\beta_2})=\underset{\stackrel{i, j>0}{i\beta_1+j\beta_2 \in \Phi}}{\prod}(1+U_{i\beta_1+j\beta_2})^{c_{ij}p^{i+j-1}}(1+U_{\beta_2})(1+U_{\beta_1}).$
      \\
 \noindent
 (\textbf{R-7})
 For $\alpha \in \Phi^{+}, \beta \in \Phi^{-}$,
    $$(1+V_{\alpha})(1+U_{\beta})=\underset{\stackrel{i, j>0}{i\alpha+j\beta \in \Phi^{-}}}{\prod}(1+U_{i\alpha+j\beta})^{c_{ij}p^{j-1}}\underset{\stackrel{i, j>0}{i\alpha+j\beta \in \Phi^{+}}}{\prod}(1+V_{i\alpha+j\beta})^{c_{ij}p^j}(1+U_{\beta})(1+V_{\alpha}).$$
  \\
 \noindent
 (\textbf{R-8})
 With $\alpha= \sum_i n_i \delta_i,$
    $(1+V_\alpha)(1+ U_{-\alpha})= (1+ U_{-\alpha})^Q\underset{\delta_i\in \Delta}{\prod}(1+ W_{\delta_i})^{n_i }(1+V_\alpha)^Q,$
    where $Q= (1+p)^{-1}.$
\section{Explicit presentation of the Iwasawa algebra}

Let $|\Phi|+|\Delta| =d$ and let us totally order $S:=\{V_\alpha,W_\delta,U_\beta,\alpha \in \Phi^+,\delta \in \Delta,\beta \in \Phi^{-}\}$ with an order as prescribed in Theorem \ref{Ordered basis}. Let  $X:=\{X_1,\cdots, X_d\}$. It is clear that there is a natural bijection between $\psi: X \ra S$ sending $X_j$ to the $j$-th element  of $S$.

Let  $\cA$ be the universal $p$-adic algebra of non-commutative power series in the variables $X_1,\cdots, X_d$ with coefficients in $\mathbb{Z}_p$ with the ordering as in set $S$. Let us note that any monomial of degree $n$ in $\cA$ can be expressed as $a_iX_{i(1)}\cdots X_{i(n)}$ where $i$ is a map $i:\{1,\cdots, n\}\ra \{1,\cdots, d\}$.  Thus, writing $X_i:=X_{i(1)}\cdots X_{i(n)}$ we can write 
\[\cA:=\big\{f=\sum_n\sum_i a_iX_i \big| \,a_i\in \Zp\big\}.\]
The topology of $\cA$ is given by the valuation 

\[v_{\cA}(f=\sum_n\sum_i a_iX_i)=\inf_i(v_p(a_i)+ |i|),\]
where $|i|=  i(1)+\cdots+ i(n)$.
The powers of the maximal ideal  $\mathfrak{m}_{\cA}:=(p,X_1,\cdots, X_d)$ gives a fundamental system of neighborhoods of $0$.

Let $\cR \subset \cA$ be the closed two-sided ideal generated in $\cA$ by the relations \textbf{(R-1)}-\textbf{(R-8)} (after we have translated them as a relation between the variables $X_i$'s using $\psi$). Let $\overline{\cA}$ be the reduction modulo $p$ of $\cA$. The same proof as \cite[lemma 1.3]{Clo11} gives us

\begin{lemma}\label{R mod p}
    Let $\overline{\cR}$ be the image of $\cR$ in $\overline{\cA}$. Then  $\overline{\cR}$ is the closed two-sided ideal in $\overline{A}$ generated by the relations \textbf{(R-1)}-\textbf{(R-8)}.
\end{lemma}

Let $A:=\Zp\{X_1,\cdots, X_d\}$ be the non-commutative  polynomial ring in the variables $X_1,\cdots, X_d$. The natural inclusion map $A\ra \cA$ has dense image.

\begin{lemma}\label{natural surjection}
  Let us (by abuse of notation) define a natural map $\psi : A \ra \Lambda(I)$ mapping $X_i \in A$ to the corresponding $i$-th element among $S:=\{V_\alpha,W_\delta,U_\beta,\alpha \in \Phi^+,\delta \in \Delta,\beta \in \Phi^{-}\}$ in the Iwasawa algebra $\Lambda(I)$. Then $\psi$ extends continuously to a surjective homomorphism $\cA \ra \Lambda(I)$.
\end{lemma}

\begin{proof} Using Lemma \ref{Zp module isomorphism of Iwasawa algebra} and identifying via $\psi$ we can write an element $\mu\in \Lambda(I)$ as $\mu =\sum_i a_i\psi (X_i)$ and we have 
\[\begin{array}{cccc}
    \tilde{\omega}_\Lambda(\mu)&=& \inf_{i}(v_p(a_i)+ \sum_{j=1}^d i(j)\omega(h_j)) \text{ }(\text{from } \eqref{eq:evaluate}), \\
     &\ge&  \inf_{i}(\frac{v_p(a_i)}{h}+ \sum_{j=1}^d i(j)\frac{1}{h}) \;(\text{as } \omega(h_d)\ge \frac{1}{h},\; \text{remark }\ref{bound on valuation of basis}),\\
     &\ge& \frac{1}{h}v_\cA(\sum_i a_i (X_i)).
\end{array}\]

Hence the map is continuous. Surjectivity is almost automatic from the definition of $\psi$.    
\end{proof}

Consider the the natural filtration of $\ovcA$ by powers of the maximal ideal $\frm_{\ovcA}$, which we denote by $F^n\ovcA$. We have $F^n\ovcA/F^{n+1}\ovcA= {\rm gr}^n\ovcA$. The filtration $F^n$ induces a filtration on $\ovcB=\ovcA/\ovcR$
\[F^n\ovcB= F^n\ovcA+\ovcR,\]
and hence a gradation
\[{\rm gr}^n\ovcB= F^n\ovcB/F^{n+1}\ovcB.\]

\begin{theorem}\label{dimension bounding}
With the notation as above, for $n \geq 0$, we have 
\[\dim_{\Fp}{\rm gr}^n\ovcB\le \dim_{\Fp} {\rm gr}^n\Omega(I).\]   
\end{theorem}

The proof of this theorem is the main heart of the paper and occupies the whole of Section \ref{bound_dimenion}.

\section{Bounding the dimension}\label{bound_dimenion}

\subsection{Plan for the proof of Theorem \ref{dimension bounding}} The theorem is obviously true for $n=0$ as $\mathrm{gr}^0$ is $\Fp$ and both sides equal $1$. For $n=1$, ${\rm gr}^1\ovcB$ is a quotient of the space 
$F^1\ovcA/F^{2}\ovcA$ with basis $\{V_\alpha, U_\beta\}$ where $\alpha$ varies over all the simple roots (i.e. $\mathrm{ht}(\alpha)=1$) and $\beta$ is the negative of the highest root (hence $h+\mathrm{ht}(\beta)=1$). Hence $$\dim_{\Fp}{\rm gr}^1\ovcB\le |\Delta|+1=\dim_{\Fp} {\rm gr}^1\Omega(I).$$
For the general case, we first reduce all the relations $\textbf{(R-1)} -\textbf{(R-8)}$ modulo appropriate filtrations (see Proposition \ref{prop:wrongorder} in Section \ref{sec:modulo_fil}). Proposition \ref{prop:wrongorder} tells us how the product of two variables in the \textit{wrong} order is related to the product of the variables in the \textit{right} order (the right order is the one prescribed by Theorem \ref{Ordered basis}).  Next we show that we can arrange the variables in the wrong order in ${\rm gr}^n\ovcB$, by a sequence of transpositions and put them in the right order, and at each step we reduce the number of inversions (see Lemma \ref{Counting inversions} in Section \ref{sec:inversion}).
This shows that $\dim_{\Fp}{\rm gr}^n\ovcB\le \dim_{\Fp} {\rm gr}^n\Omega(I)$ since 
${\rm gr}^n\Omega(I)$ consists of \textit{homogeneous symmetric polynomials}.

Let us start the proof of Theorem \ref{dimension bounding} by analyzing the mod-$p$ reduction of certain binomial coefficients which will be necessary to study the relations between the variables modulo filtrations of an appropriate degree inside $\overline{\mathcal{R}}.$
\begin{lemma}\label{p-valuation of chooses}
Let $1\le c \le p-1$, $k \geq 1$, $p^k\nmid r$  and $1\le r < p^k$ then we have 

\[{cp^k \choose r}\equiv 0 \text{ } (\mod p) \]
    
\end{lemma}

\begin{proof} For any positive integer $n$, write $n$ in base $p$, i.e. $n=\sum_{i=0}^da_ip^d,$ where $0 \leq a_i \leq p-1$ and $a_d \neq 0$. Define the function $s_p(n):=\sum_{i=0}^da_i.$ 
Let us recall $v_p(n!)=\frac{n- s_p(n)}{p-1}$. Thus we have 
\[\begin{array}{ccccccc}
   v_p({cp^k \choose r})&= \frac{cp^k- s_p(cp^k)}{p-1}- \frac{r-s_p(r)}{p-1} -\frac{cp^k- r -s_p(cp^k- r)}{p-1}  &  \\
     &=\frac{s_p(cp^k- r)+ s_p(r)- s_p(cp^k)}{p-1}
\end{array}\]
From \cite[Proposition 2.1]{HLS} we get 
\[\begin{array}{cccc}
s_p(cp^k- r)&=& s_p(c-1)+(p-1)k - s_p (r-1)\\
s_p(cp^k- r)+ s_p(r)- s_p(cp^k)&=& s_p(c-1)+(p-1)k - s_p (r-1)+s_p(r)-c\\
= c-1+ (p-1)k - s_p (r-1)+s_p(r)-c &=& (p-1)k+s_p(r)- s_p(r-1)-1
\end{array}\]
Since we assumed $r<p^k $ we have $s_p(r-1) < (p-1)k$ and of course $s_p(r) -1 \ge 0$ . Thus we have 

\[v_p({cp^k \choose r})=(p-1)k+s_p(r)- s_p(r-1)-1 > (p-1)k-(p-1)k = 0\]
And since it is an integer we get $v_p({cp^k \choose r}) \ge 1$, which is what we needed to show.

\end{proof}

\subsection{Relations modulo appropriate filtrations}\label{sec:modulo_fil}
Recall that $\deg(W_\delta)= h, \deg(V_\alpha)= \height(\alpha), \deg(U_\beta)= h+\height(\beta)$.

\vskip8pt
The variables in relations \textbf{(R-1)} and \textbf{(R-4)} commute modulo any filtration. In this following proposition we compute the other relations modulo appropriate filtrations. We will see that we are reduced to only three relations \textbf{(R-5')}, \textbf{(R-7')} and \textbf{(R-8')} given below where the variables don't commute modulo filtration of an appropriate order. 
\begin{prop}\label{prop:wrongorder} We have 

\noindent \textbf{(R-2')} $W_\delta V_\alpha =V_\alpha W_\delta \text{ } (\mod   \Fil^{\height(\alpha)+ h+1}). $

\noindent \textbf{(R-3')} $W_{\delta} U_{\beta}=U_{\beta} W_{\delta}\text{ } (\mod   \Fil^{2h+\height(\beta)+1}).$ 

\noindent \textbf{(R-5')}
$V_{\alpha_1}V_{\alpha_2}=\underset{\underset{i\alpha_1+j\alpha_2\in \Phi}{i,j>0}}{\sum}c_{ij}V_{i\alpha_1+j\alpha_2}+V_{\alpha_2}V_{\alpha_1}\text{ }(\mod \Fil^{\height(\alpha_1)+\height(\alpha_2)+1}).$

\noindent \textbf{(R-6')} $U_{\beta_1}U_{\beta_2}=U_{\beta_2}U_{\beta_1}
\text{ } (\mod \Fil^{2h+ \height(\beta_1)+\height(\beta_2)+1}).$

\noindent \textbf{(R-7')}
$V_\alpha U_\beta= c_{11}U_{\alpha+\beta}+U_\beta V_\alpha \text{ } (\mod \Fil^{h+\height(\alpha)+ \height(\beta)+1}).$

\noindent \textbf{(R-8')}
$V_\alpha U_{-\alpha} = \underset{\delta_i \in \Delta}{\sum}n_iW_{\delta_i}+U_{-\alpha}V_\alpha \text{ } (\mod \Fil^{h+1}).$
\end{prop}

\begin{proof}
Consider relation
 \textbf{(R-2)}

\[(1+ W_{\delta})(1+ V_\alpha)= (1+ V_\alpha)^M(1+ W_{\delta}), \alpha\in \Phi^{+}, M= (1+p)^{\langle \alpha, \delta\rangle}. \]

Note that $\deg(W_\delta V_\alpha)=h+\height(\alpha)$ and so we are interested in studying relation \textbf{(R-1)} modulo $\Fil^{\height(\alpha)+ h+1}$. 
 
We want to show that relation
 \textbf{(R-2)} reduces to

\[(1+ W_{\delta})(1+ V_\alpha)= (1+ V_\alpha)(1+ W_{\delta}) \text{ } (\mod   \Fil^{\height(\alpha)+ h+1}).\]

We will show that by proving
\[(1+ V_\alpha)^M\equiv (1+ V_\alpha) \text{ } (\mod \Fil^{\height(\alpha)+h+1}). \]

We note that $M\equiv 1 \text{ } (\mod p)$ and

\[(1+ V_\alpha)^M= 1 + M V_\alpha+ {M\choose 2} V_\alpha^2+\cdots\]

Since degree of $V_\alpha^m= \height(\alpha)m$ for any $m\in \bbN$ we see that $V_\alpha^m \equiv 0 \text{ } (\mod \Fil^{\height(\alpha)+h+1})$ as long as $\height(\alpha)m >\height(\alpha)+h$. So we will be done if we can show ${M \choose m} \equiv 0 \text{ } (\mod p)$ for $m\ge 2$ and $\height(\alpha)m \le\height(\alpha)+h$ or rewriting, $\height(\alpha)(m-1) \le h$ then we are done. But note that, if $\height(\alpha)(m-1)\le h$ then certainly $m-1\le h$. Thus if we can prove ${M \choose m} \equiv 0 \text{ } (\mod p)$ for $m-1\le h$ we are done. By assumption, $h< p-1$, thus if $m-1\le h < p-1 $ then $m < p$. So it's clear in that case ${M \choose m} \equiv 0 \text{ }(\mod p)$.
\vskip8pt
This proves  \textbf{(R-2')}. 

Next we want to show that relation 
 \textbf{(R-3)}
\[(1+ W_{\delta})(1+ U_\beta)= (1+ U_\beta)^M(1+ W_{\delta}), \beta\in \Phi^{-}, M= (1+p)^{\langle \beta, \delta\rangle}, \]

reduces to 

\[(1+ W_{\delta})(1+ U_\beta)= (1+ U_\beta)(1+ W_{\delta})\text{ } (\mod   \Fil^{2h+\height(\beta)+1}). \]

Just like the last relation we want to show

\[(1+ U_\beta)^M\equiv 1+ U_\beta \text{ } (\mod \Fil^{2h+\height(\beta)+1}).\]
When $(\height(\beta)+h)m > 2h+\height(\beta)$ then $U_\beta^m \equiv 0 \text{ }  (\mod  \Fil^{2h+\height(\beta)+1})$ thus we need to worry about $m\ge 2$ and $(\height(\beta)+h)m \le 2h+\height(\beta)$ or rewriting  $(\height(\beta)+h)(m-1) \le h$. But pretty much by the same argument as above we are reduced to the case $m-1\le h$ and the above calculation goes through which proves \textbf{(R-3')}.
\vskip8pt

Next consider relation \textbf{(R-6)}.
We want to show the relations,
for $\beta_1, \beta_2 \in \Phi^{-}$,
    \[(1+U_{\beta_1})(1+U_{\beta_2})=\prod_{\stackrel{i, j>0}{i\beta_1+j\beta_2 \in \Phi}}(1+U_{i\beta_1+j\beta_2})^{c_{ij}p^{i+j-1}}(1+U_{\beta_2})(1+U_{\beta_1}),\]

reduces to 

    \[(1+U_{\beta_1})(1+U_{\beta_2})=(1+U_{\beta_2})(1+U_{\beta_1})\text{ } (\mod \Fil^{2h+ \height(\beta_1)+\height(\beta_2)+1}).\]

 We need to analyze    $(1+U_{i\beta_1+j\beta_2})^{c_{ij}p^{i+j-1}}$ carefully. As usual $$\deg(U_{i\beta_1+j\beta_2}^m)= m(h+ \height(i\beta_1+j\beta_2)).$$ Thus for any $m$ such that 
 $m(h+ \height(i\beta_1+j\beta_2)) > 2h+ \height(\beta_1)+\height(\beta_2) $, we have 
 \[\deg(U_{i\beta_1+j\beta_2}^m)\equiv 0 \text{ }(\mod \Fil^{2h+ \height(\beta_1)+\height(\beta_2)}+1).\]
From now on we safely make the assumption,  

\[\begin{array}{ccccc}
   m(h+ \height(i\beta_1+j\beta_2))    &\le& 2h+ \height(\beta_1)+\height(\beta_2) \\
  \text{gives}\;   m &\le& \frac{2h+ \height(\beta_1)+\height(\beta_2)}{h+ \height(i\beta_1+j\beta_2)}\\
  &\le& 2h -2 \;(\text{as } h+ \height(i\beta_1+j\beta_2)\ge 1)\\
  &\le& 2(p-1)-2= 2p-4 \;(\text{as } p>1+h).
\end{array}\]

For any prime $p\ge 3$ and any $k\ge 2$ we have $p^k > 2p- 4$. 
Thus we have
\[{c_{ij}p^{i+j-1}\choose m} \equiv 0 \text{ } (\mod p),\; m\le 2p-4 \;\text{and}\;i+j-1 \ge 2.\]
We are left to deal with the case $i+j= 2$ or $i=1, j=1$. Note that by Lemma \ref{p-valuation of chooses}, ${c_{ij}p\choose m} \equiv 0 \text{ } (\mod p)$ for $m < p$. Thus we can make the assumption $m \ge p $. But in this case we have
\[\begin{array}{ccccccc}
    m(h+ \height(\beta_1+\beta_2)) &\ge& p (h+ \height(\beta_1+\beta_2))\\
    &>& (1+h) (h+ \height(\beta_1+\beta_2)) = h+ \height(\beta_1+\beta_2)+ h(h+ \height(\beta_1+\beta_2))\\ &\ge& h+ \height(\beta_1+\beta_2) + h \text{ }(\text{as } h+ \height(\beta_1+\beta_2)\ge 1) \\
    &=& 2h+ \height(\beta_1+\beta_2).
\end{array}\]
Hence this proves \textbf{(R-6')}.

\vskip8pt

Next consider relation \textbf{(R-5)}. We want to show the relation, $\alpha_1, \alpha_2 \in \Phi^{+}$,
    \[(1+V_{\alpha_1})(1+V_{\alpha_2})=\prod_{\stackrel{i, j>0}{i\alpha_1+j\alpha_2 \in \Phi}}(1+V_{i\alpha_1+j\alpha_2})^{c_{ij}}(1+V_{\alpha_2})(1+V_{\alpha_1}),\]
 reduces to 
    \[V_{\alpha_1}V_{\alpha_2}=\sum_{\stackrel{i, j>0}{i\alpha_1+j\alpha_2 \in \Phi}}c_{ij}V_{i\alpha_1+j\alpha_2}+V_{\alpha_2}V_{\alpha_1}\text{ } (\mod \Fil^{\height(\alpha_1)+\height(\alpha_2)+1}).\]

This follows from the observations that 

\[(1+V_{i\alpha_1+j\alpha_2})^{c_{ij}}\equiv 1+ c_{ij}V_{i\alpha_1+j\alpha_2} \text{ } (\mod \Fil^{\height(\alpha_1)+\height(\alpha_2)+1})\]
by degree calculation and that $\deg(V_{i\alpha_1+j\alpha_2}) + \deg(V_{\alpha_i}) > \height(\alpha_1)+\height(\alpha_2) $ for $i\in \{1, 2\}$.
This shows \textbf{(R-5')}.
\vskip8pt

Next consider the relation \textbf{(R-7)}. For $\alpha \in \Phi^{+}, \beta \in \Phi^{-}$,
\[(1+V_{\alpha})(1+U_{\beta})=\prod_{\stackrel{i, j>0}{i\alpha+j\beta \in \Phi^{-}}}(1+U_{i\alpha+j\beta})^{c_{ij}p^{j-1}}\prod_{\stackrel{i, j>0}{i\alpha+j\beta \in \Phi^{+}}}(1+V_{i\alpha+j\beta})^{c_{ij}p^j}(1+U_{\beta})(1+V_{\alpha}).\]

\vskip8pt

Let us first deal with the case $i\alpha+ j\beta \in \Phi^{-}$. For $ \deg U_{i\alpha+j\beta}^m=m(h+ \height(i\alpha+j\beta))   > h+ \height(\beta)+\height(\alpha)$, we have $ U_{i\alpha+j\beta}^m=0 \text{ }(\mod \Fil^{h+ \height(\beta)+\height(\alpha)+1}). $ Hence we can restrict to the case 
$m(h+ \height(i\alpha+j\beta))   \leq h+ \height(\beta)+\height(\alpha).$
\vskip8pt

\[\begin{array}{ccccc}
   m(h+ \height(i\alpha+j\beta))    &\le& h+ \height(\beta)+\height(\alpha) \\
  \text{gives}\;   m &\le& \frac{h+ \height(\beta)+\height(\alpha)}{h+ \height(i\alpha+j\beta)}\\
  &\le& 2h -2 \;(\text{as } h+ \height(i\alpha+j\beta)\ge 1, \height(\beta) \leq -1, \height(\alpha)\leq h-1)\\
  &\le& 2(p-1)-2= 2p-4 \;(\text{as } p>1+h).
\end{array}\]

By Lemma \ref{p-valuation of chooses}, similar to the argument we gave while deducing \textbf{(R-6')}, we see that if $j-1 \ge 2$, then with this restriction on $m$, none of the corresponding terms will contribute modulo the filtration as the coefficients become trivial modulo $p$.

\vskip8pt

When $j-1=1$ or $j=2$ then we potentially will have contribution from $m=p$ term. In that case we have,
 $\deg(U_{i\alpha+ 2\beta})^p= p(h+ \height(i\alpha+2\beta))> (h+1)(h+ \height(i\alpha+2\beta))$.

\vskip8pt
 Note that as $h+\height(i\alpha+2\beta)>1$, $(h+1)(h+ \height(i\alpha+2\beta))\geq 2(\textit{h}+1)$. So, of course,
 $(h+1)(h+ \height(i\alpha+2\beta))> \textit{h}+ \height(\alpha)+ \height(\beta)$ if $\height(\alpha) +\height(\beta) < 1$.  So we assume

 $\height(\alpha)+ \height(\beta) \ge 1$

\[\begin{array}{ccccc}
  \height(\alpha)+ \height(\beta) &\ge& 1 \\
\text{hence }  \text{   } \text{   } i\height(\alpha) &\ge& i(1-\height(\beta))\\
\text{which gives } \text{   } \text{   } i\height(\alpha)+ 2\height(\beta) &\ge& i(1-\height(\beta))+2\height(\beta)= \textit{i}+ (2-\textit{i})\height(\beta).
\end{array}\]
So if $2-i\le 0$ then the right hand side is definitely $>0$, but this is not possible as $i\alpha+2\beta  $ was assumed to be a negative root. Thus we see that only value $i$ can take is $1$. But in that case
as $p>1+h$, 
\begin{align*}
  p(h + \height(\alpha)+ 2\height(\beta))&>(h+1)(h + \height(\alpha)+ 2\height(\beta))  \\
  &=(h + \height(\alpha)+\height(\beta)) + h\big(h + \height(\alpha)+ 2\height(\beta)\big)+\height(\beta)\\
  &=(h + \height(\alpha)+\height(\beta)) + h\big(h + \height(\alpha+2\beta)\big)+\height(\beta)\\
  &>(h + \height(\alpha)+\height(\beta))+(h+\height(\beta))\text{ } (\text{as } h+\height(\alpha+2\beta) >1)\\
  &>h + \height(\alpha)+\height(\beta) \text{ } (\text{as } h+\height(\beta) >1).
\end{align*}

When $j=1$ then 
 $m(h+ \height(i\alpha+\beta)) \geq   \textit{h}+ \height(\alpha)+\height(\beta)$ for all $ m$, with equality only when $m=1, i=1$. Thus the only contribution we can possibly have is $(1+c_{11} U_{\alpha+\beta})$.

 \vskip8pt

 Now to look at the contribution coming from positive roots part of the equation \textbf{(R-7)}, i.e. from the terms
 $\underset{\underset{i\alpha+j\beta \in \Phi^{+}}{i, j>0}}{\prod}(1+V_{i\alpha+j\beta})^{c_{ij}p^j}.$

By the same logic as above (i.e. bounding $m$ by $2p-4$, then showing that the binomial coefficients become trivial modulo $p$) this time also $j\ge 2$ gets ruled out immediately. The calculation is similar, hence omitted.
\vskip8pt 

When $j=1$ we are again concerned with the $m=p$ term only.

In this case we have 
\begin{align*}
\deg(V_{i\alpha+ \beta})^p= p\big( \height(i\alpha+\beta)\big)&>(h+1) \big(i\height(\alpha)+\height(\beta)\Big)\\
&=i\height(\alpha) +\height(\beta) +h \big(\height(i\alpha+\beta)\big)\\
&\geq h+\height\alpha+\height(\beta) \text{ } (\text{as } i>0, \height(i\alpha+\beta)>0).
\end{align*}

Thus relation \textbf{(R-7)} reduce to 

\[ (1+V_\alpha)(1+ U_\beta)= (1+c_{11} U_{\alpha+\beta})(1+ U_\beta)(1+ V_\alpha) \text{ }(\mod \Fil^{h+ \height(\beta)+\height(\alpha)+1}), \;\text{when}\;\alpha+\beta\in \Phi^{-}.\]

Now, $\deg(U_{\alpha+\beta}V_\alpha)=h+2\height(\alpha)+\height(\beta)>h+ \height(\beta)+\height(\alpha)$ and
\begin{align*}
\deg(U_{\alpha+\beta}U_\beta)&=2h+\height(\alpha) +2\height(\beta)\\
&=h+ \height(\beta)+\height(\alpha)+(h+\height(\beta))\\
&>h+ \height(\beta)+\height(\alpha) \text{ }(\text{as } h+\height(\beta) >1).
\end{align*}
hence, after further simplifying we obtain

 \[V_\alpha U_\beta= c_{11}U_{\alpha+\beta}+U_\beta V_\alpha\text{ }(\mod \Fil^{h+ \height(\beta)+\height(\alpha)+1})\]
 which is exactly relation \textbf{(R-7')}.
\end{proof}
Now consider relation \textbf{(R-8)}. First, we will show that  $(1+U_{-\alpha})^Q=(1+U_{-\alpha})\text{ }(\mod \Fil^{h+1})$. If $(h+\height(-\alpha))m >h$ then $U_{-\alpha}^m=0\text{ }(\mod \Fil^{h+1}).$ So we need to worry about the case $m \geq 2$ and $(h+\height(-\alpha))m \leq h$. But then again $m \leq h<p-1$ and so ${Q \choose m} =0\text{ } (\mod p)$ since $Q=1\text{ } (\mod p)$.

Next, we will show that  $(1+V_{\alpha})^Q=(1+V_{\alpha})\text{ }(\mod \Fil^{h+1})$.
If $m\height(\alpha) >h$, the the degree of $V_{\alpha}^m$ is strictly larger than $h$. If  $m\height(\alpha) \leq h$, then $m \leq h< p-1$ and so again ${Q \choose m} =0\text{ } (\mod p)$.
Hence relation \textbf{(R-8)} reduces to
$$(1+V_\alpha)(1+ U_{-\alpha})= (1+ U_{-\alpha})\underset{\delta_i\in \Delta}{\prod}(1+ W_{\delta_i})^{n_i }(1+V_\alpha)\text{ }(\mod \Fil^{h+1}).$$
Since $\deg(W_{\delta_i})=h$ and $ n_i \geq 0$, further simplification yields $$V_\alpha U_{-\alpha} = \underset{\delta_i \in \Delta}{\sum}n_iW_{\delta_i}+U_{-\alpha}V_\alpha \text{ } (\mod \Fil^{h+1}).$$

\subsection{Computing inversions modulo filtrations}\label{sec:inversion}

   Recall that $\{X_1,\cdots, X_d \}$ was an ordered set of variables generating the noncommutative $p$-adic power series ring $\mathcal{A}$. We set $X^i:=X_{i_1}\cdots X_{i_w}$ to mean the monomial $X_{i(1)}\cdots X_{i(w)}$. Then we can change $X^i$ to a well-ordered monomial ($b \rightarrow i_b$ increasing) by a sequence of transpositions. Consider a move $(b,b+1) \rightarrow (b+1,b)$ and assume $i_b >i_{b+1}.$ We write $$X^i=X^fX_bX_{b+1}X^e,$$ and we look at $X^i \text{ } (\mod \Fil^{\deg(X^i)+1}).$ Suppose   $(X_b,X_{b+1})$ is any pair of variables among $\{U_\beta, W_{\delta_i}, V_\alpha, \beta \in \Phi^-, \delta_i \in \Delta, \alpha \in \Phi^+\}$ but in the wrong ordering (for example, let $(X_b,X_{b+1})=(W_\delta,V_\alpha)$ from relation $\textbf{(R-2')}$).
    From Proposition \ref{prop:wrongorder}, we see that variables in the wrong ordering \textit{commute module filtration} of an appropriate degree except relations \textbf{(R-5')}, \textbf{(R-7')} and \textbf{(R-8')}. Hence, except for these three relations, $$X_bX_{b+1}=X_{b+1}X_b\text{ } (\mod \Fil^{\deg(X_b)+\deg(X_{b+1})+1}).$$ Therefore, we can reduce the number of inversions and change the wrong ordered variables into right order, modulo suitable filtration.
    For the three relations \textbf{(R-5')}, \textbf{(R-7')} and \textbf{(R-8')}, we have to compute the number of inversions explicitly.
     
 \begin{lemma}\label{Counting inversions}
        Suppose $(X_b,X_{b+1})$ be either of the three pairs  $(V_{\alpha_1},V_{\alpha_2}),$  $(V_\alpha,U_\beta)$ or $(V_\alpha,U_{-\alpha})$ in the wrong ordering  coming from left hand side of the relations \textbf{(R-5')}, \textbf{(R-7')} and \textbf{(R-8')} satisfying equations of the form 
      $$X_bX_{b+1}=X_{b+1}X_b+ \text{ linear terms in } X_i \text{ } (\mod \Fil^{\deg(X_b)+\deg(X_{b+1})+1}).$$
      Then the number of inversions in $X^fX_{b+1}X_bX^e$ and $X^fX_iX^e$ are both strictly less than the number of inversions of $X^i$. That is, the two moves $X_{b}X_{b+1} \rightarrow X_{b+1}X_{b},$ and $X_{b}X_{b+1} \rightarrow X_i$ reduces the number of inversions modulo $\Fil^{\deg(X_b)+\deg(X_{b+1})+1}$.
    \end{lemma}
  \begin{proof}
 First let us compute the number of inversions of $X^i$.
Let $\xi$ and $\mu$ be denote indices $\neq b, b+1$ in the product $X^i$. 

Consider relation \textbf{(R-5')}.
The number of inversions in $X^i$ was originally as follows:

$$\mathrm{inv}=\underset{\underset{i_\xi>i_\mu}{\xi<\mu}}{\sum}1+\underset{\underset{i_\mu \in [\mathrm{I}(V_{\alpha_2}),\mathrm{I}(V_{\alpha_1}))}{\mu>b+1}}{\sum}1+\Big(2\underset{\underset{i_\mu<\mathrm{I}(V_{\alpha_2})}{\mu>b+1}}{\sum}1\Big)+\underset{\underset{i_\xi \in (\mathrm{I}(V_{\alpha_2}),\mathrm{I}(V_{\alpha_1})]}{\xi<b}}{\sum}1+\Big(2\underset{\underset{i_\xi>\mathrm{I}(V_{\alpha_1})}{\xi<b}}{\sum}1\Big)+1,$$

where $\mathrm{I}(*):=\mathrm{index}(*)$ is the index of the variable $*$.

As $\I(V_{\alpha_2}) <\I(V_{\alpha_1})$, the transition $V_{\alpha_1}V_{\alpha_2} \rightarrow V_{\alpha_2}V_{\alpha_1}$ clearly reduces the number of inversions, and by changing $V_{\alpha_1}V_{\alpha_2}$ to $c_{ij}V_{i{\alpha_1}+j{\alpha_2}}$ we get

$${\mathrm{inv}^{\prime}}=\underset{\underset{i_\xi>i_\mu}{\xi<\mu}}{\sum}1+\underset{\underset{i_\mu<\mathrm{I}(V_{i{\alpha_1}+j{\alpha_2}})}{\mu>b+1}}{\sum}1+\underset{\underset{i_\xi>\mathrm{I}(V_{i{\alpha_1}+j{\alpha_2}})}{\xi<b}}{\sum}1$$

Because of the choice of an \textit{additive} ordering on the set of positive roots in Section \ref{sec:additive}, 
\begin{equation}\label{eq:keyfact}
 \I(V_{\alpha_2}) < \I(V_{i{\alpha_1}+j{\alpha_2}}) <\I(V_{\alpha_1}).\end{equation}

\vspace{.2cm}
Hence the second term of  ${\mathrm{inv}^{\prime}}$ is 
\begin{align*}
\underset{\underset{i_\mu<\mathrm{I}(V_{i{\alpha_1}+j{\alpha_2}})}{\mu>b+1}}{\sum}1 \leq \underset{\underset{i_\mu<\mathrm{I}(V_{{\alpha_1}})}{\mu>b+1}}{\sum}1=\underset{\underset{i_\mu \in [\mathrm{I}(V_{\alpha_2}),\mathrm{I}(V_{\alpha_1}))}{\mu>b+1}}{\sum}1+\underset{\underset{i_\mu<\mathrm{I}(V_{\alpha_2})}{\mu>b+1}}{\sum}1
\leq \underset{\underset{i_\mu \in [\mathrm{I}(V_{\alpha_2}),\mathrm{I}(V_{\alpha_1}))}{\mu>b+1}}{\sum}1+\Big(2\underset{\underset{i_\mu<\mathrm{I}(V_{\alpha_2})}{\mu>b+1}}{\sum}1\Big),
\end{align*}
and the third term of ${\mathrm{inv}^{\prime}}$ is
\begin{align*}
\underset{\underset{i_\xi>\mathrm{I}(V_{i{\alpha_1}+j{\alpha_2}})}{\xi<b}}{\sum}1 \leq \underset{\underset{i_\xi>\mathrm{I}(V_{{\alpha_2}})}{\xi<b}}{\sum}1 = \underset{\underset{i_\xi \in (\mathrm{I}(V_{\alpha_2}),\mathrm{I}(V_{\alpha_1})]}{\xi<b}}{\sum}1+\underset{\underset{i_\xi>\mathrm{I}(V_{\alpha_1})}{\xi<b}}{\sum}1 \leq \underset{\underset{i_\xi \in (\mathrm{I}(V_{\alpha_2}),\mathrm{I}(V_{\alpha_1})]}{\xi<b}}{\sum}1+\Big(2\underset{\underset{i_\xi>\mathrm{I}(V_{\alpha_1})}{\xi<b}}{\sum}1\Big).
\end{align*}
Hence ${\mathrm{inv}^{\prime}}<\mathrm{inv}.$ for the relation \textbf{(R-5')}.

The proofs for the relations \textbf{(R-7')} and \textbf{(R-8')} are similar and hence omitted, the key necessary things corresponding to \eqref{eq:keyfact} being 
$\I(U_\beta) <\I(U_{\alpha+\beta}) <\I(V_\alpha)$ for the relation \textbf{(R-7')} 
and 
$\I(U_{-\alpha}) <\I(W_{\delta_i}) <\I(V_\alpha)$ for the relation \textbf{(R-8')}.
  \end{proof}

\section{Proof of the main theorem}
In this section we recall and prove the main theorem in two steps. Both the proofs are exactly same proof as that of Clozel (see \cite[pg. 554,  Theorems 1.2 and 1.4]{Clo11}), we recall them here for the purpose of completion.

\begin{theorem}\label{Main isomorphism of mod p iwasawa algebras}
The mod-$p$ Iwasawa algebra $\Omega(I)$ is naturally isomorphic to $\ovcA/\ovcR$.
    
\end{theorem}

\begin{proof}
The natural map $\psi : \cA \ra \Lambda(I)$ sends $\cM^n_{\cA}$ to $\cM^n_{\Lambda(I)}$, where $\cM_{\Lambda(I)}$ is the maximal ideal of $\Lambda(I)$.  The reduction mod-$p$ of $\psi$ respects the natural filtration on both sides and thus induce a natural map $\gr\psi: \gr\cA\ra \gr\Lambda(I)$. The induced map on the associated graded is  surjective since $\psi$ is surjective. Thus it is an isomorphism by Theorem  \ref{dimension bounding}. This then imply the theorem since the filtration on $\cB$ is complete.
\end{proof}

\begin{theorem}\label{Main isomorphism of iwasawa algebras}
The Iwasawa algebra $\Lambda(I)$ is naturally isomorphic to $\cA/\cR$.
    
\end{theorem}
We replicate the proof which is also the same argument as Clozel's corresponding statement.
\begin{proof}
    The reduction of $\psi: \cA/\cR \ra \Lambda(I)$ is $\overline{\psi}$. Recall that $\ovcR$ is the image of $\cR$ in $\ovcA$. Assume $f\in \cA$ satisfies $\psi(f)=0$. We then have $\overline{f}\in \overline{\cR}$ by Theorem \ref{Main isomorphism of mod p iwasawa algebras}. So $f = r_ 1+ p f_1$, $r_1 \in \cR, f_1 \in \cA$. Then by applying $\psi$ to both sides we get $\psi(f_1) = 0$. Inductively, we obtain an expression $f = r^n + p^nf_n$ with $r_n\in \cR$ and $f_n\in \cA$. Since $p^nf_n$ tends to  $0$ in $\cA$ and $\cR$ is closed, we see that $f \in \cR$.
\end{proof}

\vskip8pt
\bibliographystyle{alpha}
\bibliography{mybib}
\end{document}